\documentclass{amsart}
\usepackage{amscd,thmdefs,amsmath,amssymb,latexsym}
\usepackage{textcomp}
\usepackage{color}
\usepackage{hyperref}
\usepackage{tikz,pgfplots,pgf} 
\usetikzlibrary{calc} 
\usetikzlibrary{trees} 

\def\N{\mathbb{N}}
\def\R{\mathbb{R}}

\def\K{\mathbb{K}}

\def\F{\mathcal{F}}

\def\L{\mathcal{L}}

\def\lin{\mathrm{lin}}

\def\rank{\mathrm{rank}}

\def\diam{\mathrm{diam}}

\def\Lip{\mathrm{Lip}}
\def\Lipo{\mathrm{Lip}_0}
\def\LipoF{\mathrm{Lip}_{0F}}

\begin{document}
\title[The approximation property for spaces of Lipschitz functions]{The approximation property for spaces of Lipschitz functions}

\author{A. Jim\'{e}nez-Vargas}
\address{Departamento de Matem\'{a}ticas, Universidad de Almer\'{i}a, 04120 Almer\'{i}a, Spain}
\email{ajimenez@ual.es}
\thanks{This research was partially supported by Junta de Andaluc\'{i}a grant FQM-194.}



\begin{abstract}
Let $\Lipo(X)$ be the space of all Lipschitz scalar-valued functions on a pointed metric space $X$. We characterize the approximation property for $\Lipo(X)$ with the bounded weak* topology using as tools the tensor product, the $\epsilon$-product and the linearization of Lipschitz mappings. 
\end{abstract}
\maketitle

\section*{Introduction}

Let $(X,d)$ be a pointed metric space with a base point which we always will denote by $0$ and let $F$ be a Banach space. The space $\Lipo(X,F)$ is the Banach space of all Lipschitz mappings $f$ from $X$ to $F$ that vanish at $0$, with the Lipschitz norm defined by 
$$
\Lip(f)=\sup\left\{\frac{\left\|f(x)-f(y)\right\|}{d(x,y)}\colon x,y\in X, \; x\neq y\right\}.
$$ 
The elements of $\Lipo(X,F)$ are frequently called Lipschitz operators. If $\K$ is the field of real or complex numbers, $\Lipo(X,\mathbb{K})$ is denoted by $\Lipo(X)$. 
The closed linear subspace of the dual of $\Lipo(X)$ spanned by the functionals $\delta_x$ on $\Lipo(X)$ with $x\in X$, given by $\delta_x(f)=f(x)$, is a predual of $\Lipo(X)$. 
This predual is called the Lipschitz-free space over $X$ and denoted by $\F(X)$ in \cite{gk}. We refer the reader to Weaver's book \cite{weaver} for the basic theory of $\Lipo(X)$ and its predual $\F(X)$, which is called the Arens--Eells space of $X$ and denoted by $\AE(X)$ there.

The study of the approximation property is a topic of interest for many researchers. Let us recall that a Banach space $E$ has the approximation property (in short, (AP)) if for each compact set $K\subset E$ and each $\varepsilon>0$, there exists a bounded finite-rank linear operator $T\colon E\to E$ such that $\sup_{x\in K}\left\|T(x)-x\right\|<\varepsilon$. If $\left\|T\right\|\leq\lambda$ for some $\lambda\geq 1$, it is said that $E$ has the $\lambda$-bounded approximation property (in short, $\lambda$-(BAP)).


To our knowledge, little is known about the (AP) for $\Lipo(X)$. Johnson \cite{j75} observed that if $X$ is the closed unit ball of Enflo's space \cite{e}, then $\Lipo(X)$ fails the (AP). 
Godefroy and Ozawa \cite{go} showed that there exists a compact pointed metric space $X$ such that $\F(X)$ fails the (AP) and hence so does $\Lipo(X)$. For positive results, $\Lip[0,1]$ is isomorphic to $L^\infty[0,1]$ (see \cite[p. 224]{h}) and thus $\Lip[0,1]$ has the (AP). If $(X,d)$ is a doubling compact pointed metric space, in particular a compact subset of a finite dimensional Banach space, and $X^{(\alpha)}$ with $\alpha\in (0,1)$ denotes the metric space $(X,d^\alpha)$, then the space $\Lipo(X^{(\alpha)})$ is isomorphic to $\ell_\infty$ by \cite[Theorem 6.5]{k04}, and therefore $\Lipo(X^{(\alpha)})$ has the (AP). 
In \cite{j70} (see also \cite{jsv}), Johnson proved that $\Lipo(X)$ has the (AP) if and only if, for each Banach space $F$, every Lipschitz compact operator from $X$ to $F$ can be approximated in the Lipschitz norm by Lipschitz finite-rank operators. 

The most recent research on the (AP) has been directed toward $\F(X)$ rather than on $\Lipo(X)$. 
Godefroy and Kalton \cite{gk} proved that a Banach space $E$ has the $\lambda$-(BAP) if and only if $\F(E)$ has the same property. Lancien and Perneck\'{a} \cite{lp} showed that $\F(X)$ has the $\lambda$-(BAP) whenever $X$ is a doubling metric space. For stronger approximation properties as the existence of finite-dimensional Schauder decompositions or Schauder bases for certain spaces $\F(E)$, one can see the papers of Borel-Mathurin \cite{b}, Lancien and Perneck\'{a} \cite{lp} and H\'{a}jek and Perneck\'{a} \cite{hp}.
The results in those works provide apparently a limited information about the (AP) for $\Lipo(X)$ since the (AP) of a Banach space follows from the (AP) of its dual space but the converse does not always hold. 

Our aim in this paper is to study the (AP) for the space $\Lipo(X)$, with the bounded weak* topology. In the seminal paper \cite{as}, Aron and Schottenloher initiated the investigation about the (AP) for spaces of holomorphic mappings on Banach spaces. Mujica \cite{m} extended this study to the preduals of such spaces. Their techniques, based on the tensor product, the $\epsilon$-product and the linearization of holomorphic mappings, work just as well for spaces of Lipschitz mappings. 

We now describe the contents of this paper. In Section \ref{1}, we briefly recall the compact-open topology $\tau_0$, the approximation property, the $\epsilon$-product and the linearization of Lipschitz mappings. 

We address the study of the topology of bounded compact convergence $\tau_\gamma$ on $\Lipo(X)$ in Section \ref{2}. 
In the terminology of Cooper \cite{c}, we prove that $\tau_\gamma$ is the mixed topology $\gamma[\Lip,\tau_0]$ and $(\Lipo(X),\tau_\gamma)$ is a Saks space. Furthermore, it is shown that $\tau_\gamma$ agrees with the bounded weak* topology $\tau_{bw^*}$. 

We give a pair of descriptions of $\tau_\gamma$ by means of seminorms in Section \ref{3}. Assuming $X$ is compact, we first identify $\tau_\gamma$ with the classical strict topology $\beta$ introduced by Buck \cite{bu}. A second, and perhaps more interesting, seminorm description for $\tau_{\gamma}$ motivates the introduction of a new locally convex topology $\gamma\tau_\gamma$ on $\Lipo(X,F)$. 

Section \ref{4} deals with the (AP) for $(\Lipo(X),\tau_\gamma)$. We identify topologically the space $(\Lipo(X,F),\gamma\tau_\gamma)$ with the $\epsilon$-product of $(\Lipo(X),\tau_\gamma)$ and $F$, and this permits us to prove that the following properties are equivalent:
\begin{enumerate}
	\item $(\Lipo(X),\tau_\gamma)$ has the (AP).
	\item Every Lipschitz operator from $X$ into $F$ can be approximated by Lipschitz finite-rank operators within the topology $\gamma\tau_\gamma$ for all Banach spaces $F$.
	\item $\F(X)$ has the (AP).
\end{enumerate}

In Section \ref{5}, we establish a representation of the dual space of $(\Lipo(X,F),\gamma\tau_\gamma)$.

\section{Preliminaries}\label{1}

\subsection*{Topologies on spaces of Lipschitz functions}

Let $X$ be a pointed metric space and let $E$ be a Banach space. The compact-open topology or topology of compact convergence on $\Lipo(X,E)$ is the locally convex topology generated by the seminorms of the form 
$$
\left|f\right|_K=\sup_{x\in K}\left\|f(x)\right\|,\qquad f\in\Lipo(X,E),
$$
where $K$ varies over the family of all compact subsets of $X$. We denote by $\tau_0$ the compact-open topology on $\Lipo(X,E)$, or on any vector subspace of $\Lipo(X,E)$. 

The topology of pointwise convergence on $\Lipo(X,E)$ is the locally convex topology $\tau_p$ generated by the seminorms of the form 
$$
\left|f\right|_F=\sup_{x\in F}\left\|f(x)\right\|,\qquad f\in\Lipo(X,E),
$$
where $F$ ranges over the family of all finite subsets of $X$. 

Finally, we denote by $\tau_{\Lip}$ the topology on $\Lipo(X,E)$ generated by the Lipschitz norm $\Lip$. It is clear that $\tau_p\subset\tau_0$, and the inclusion $\tau_0\subset\tau_{\Lip}$ follows easily since $\left|f\right|_K\leq\Lip(f)\diam(K\cup\{0\})$ for all $f\in\Lipo(X)$ and each compact set $K\subset X$. 

\subsection*{Approximation property and $\epsilon$-product}

Let $E$ and $F$ be locally convex Hausdorff spaces. Let $\L(E;F)$ denote the vector space of all continuous linear mappings from $E$ into $F$, let $\L_b(E;F)$ denote the vector space $\L(E;F)$ with the topology of uniform convergence on the bounded subsets of $E$ and let $\L_c(E;F)$ denote the vector space $\L(E;F)$ with the topology of uniform convergence on the convex balanced compact subsets of $E$. That last topology coincides with the compact-open topology if the closed convex hull of each compact subset of $E$ is compact (for example, if $E$ is quasi-complete). When $F=\K$, we write $E'$ instead of $\L(E;\K$), $E'_b$ in place of $\L_b(E;\K)$, and $E'_c$ instead of $\L_c(E;\K)$. Unless stated otherwise, if $E$ and $F$ are normed spaces, $\L(E;F)$ is endowed with its natural norm topology. Let $E\otimes F$ denote the tensor product of $E$ and $F$, and $E'\otimes F$ can be identified with the subspace of all finite-rank mappings in $\L(E;F)$.

A locally convex space $E$ is said to have the approximation property (in short, (AP)) if the identity mapping on $E$ lies in the closure of $E'\otimes E$ in $\L_c(E;E)$. This is Schwartz's definition of the (AP) in \cite{s1}, which is slightly different from Grothendieck's definition in \cite{g}, though both definitions coincide for quasi-complete locally convex spaces. 

The $\epsilon$-product of $E$ and $F$, denoted by $E\epsilon F$ and introduced by Schwartz \cite{s1,s}, is the space $\L_\epsilon(F'_c;E)$, that is the vector space $\L_\epsilon(F'_c;E)$, with the topology of uniform convergence on the equicontinuous subsets of $F'$. Notice that if $F$ is a normed space, then equicontinuous sets and norm bounded sets in $F'$ coincide. The topology on $\L_\epsilon(F'_c;E)$ is generated by the seminorms 
$$
\alpha\epsilon\beta(T)
=\sup\left\{\left|\left\langle T(\mu),\nu\right\rangle\right|\colon \mu\in F', \; |\mu|\leq\alpha, \; \nu\in E', \; |\nu|\leq\beta\right\}, 
\qquad T\in\L_\epsilon(F'_c;E),
$$
where $\alpha$ ranges over the continuous seminorms on $F$ and $\beta$ over the continuous seminorms on $E$. 

We will use the subsequent results which follow from results of Grothendieck \cite{g}, Schwartz \cite{s1} and Bierstedt and Meise \cite{bm}. 

\begin{proposition}\cite{s1}\label{r0}
Let $E$ and $F$ be locally convex spaces. Then the transpose mapping $T\mapsto T^t$ from $E\epsilon F$ to $F\epsilon E$ is a topological isomorphism.
\end{proposition}

\begin{theorem}\cite{bm,g,s1}\label{r00}
A locally convex space $E$ has the (AP) if and only if $E\otimes F$ is dense in $E\epsilon F$ for every Banach space $F$.
\end{theorem}

\begin{proposition}\cite{bm,g,s1}\label{r000}
A locally convex space $E$ has the (AP) if $E'_c$ has the (AP).
\end{proposition}

Detailed proofs of the preceding results can be found in the paper \cite{dm} by Dineen and Mujica.

\subsection*{Linearization of Lipschitz mappings}

The study of the preduals of $\Lipo(X)$ was approached by Weaver \cite{weaver} by using a procedure to linearize Lipschitz mappings. A similar process of linearization was presented by Mujica for bounded holomorphic mappings on Banach spaces in \cite{m}. 

\begin{theorem}\label{mainth}\cite{weaver}
Let $X$ be a pointed metric space. Then there exist a unique, up to an isometric isomorphism, Banach space $\F(X)$ and an isometric embedding $\delta_X\colon X\to\F(X)$ such that 
\begin{enumerate}
\item $\F(X)$ is the closed linear hull in $\Lipo(X)'$ of the evaluation functionals $\delta_x\colon\Lipo(X)\to\mathbb{K}$ with $x\in X$, where $\delta_x(g)=g(x)$ for all $g\in\Lipo(X)$.
\item The Dirac map $\delta_X\colon X\to\F(X)$ is the map given by $\delta_X(x)=\delta_x$.
\item For each Banach space $E$ and each $f\in\Lipo(X,E)$, there is a unique operator $T_f\in\L(\F(X);E)$ such that $T_f\circ\delta_X=f$. Furthermore, $\left\|T_f\right\|=\Lip(f)$.
\item The evaluation map $f\mapsto T_f$ from $\Lipo(X,E)$ to $\L(\F(X);E)$, defined by $T_f(\varphi)=\varphi(f)$, is an isometric isomorphism.
\item $\Lipo(X)$ is isometrically isomorphic to $\F(X)'$ via the evaluation map.
\item $\F(X)$ coincides with the space of all linear functionals $\varphi$ on $\Lipo(X)$ such that the restriction of $\varphi$ to the closed unit ball $B_{\Lipo(X)}$ of $\Lipo(X)$ is continuous when $B_{\Lipo(X)}$ is equipped with the topology of pointwise convergence $\tau_p$, and hence with the compact-open topology $\tau_0$.
\end{enumerate}
\end{theorem}

The statements (i)--(v) of Theorem \ref{mainth} were proved by Weaver (see \cite[Theorem 2.2.4]{weaver}). The statement (vi) was stated in \cite[Lemma 1.1]{jsv} for $\tau_p$ and recall that, by \cite[p. 232]{k}, the topology $\tau_p$ agrees with $\tau_0$ on the equicontinuous subsets of $\Lipo(X)$, and in particular on  $B_{\Lipo(X)}$.

Viewing $\Lipo(X)$ as the dual of $\F(X)$, we can consider its weak* topology. We recall that the weak* topology on $\Lipo(X)$ is the locally convex topology $\tau_{w^*}$ generated by the seminorms of the form 
$$
p_G(f)=\sup_{\varphi\in G}\left|\varphi(f)\right|,\qquad f\in\Lipo(X),
$$
where $G$ ranges over the family of all finite subsets of $\F(X)$. 
Let us recall that $\tau_{w^*}$ is the smallest topology for $\Lipo(X)$ such that, for each $\varphi\in\F(X)$, the linear functional $f\mapsto\varphi(f)$ on $\Lipo(X)$ is continuous with respect to $\tau_{w^*}$. 

It is easy to check that $\tau_p\subset\tau_{w^*}\subset\tau_{\Lip}$. Indeed, on a hand, if $F$ is a finite subset of $X$, then $G=\delta_X(F)$ is a finite subset of $\F(X)$ and 
$$
\left|f\right|_F=\sup_{x\in F}\left|f(x)\right|=\sup_{x\in F}\left|\delta_x(f)\right|=\sup_{\varphi\in G}\left|\varphi(f)\right|=p_G(f)
$$
for all $f\in\Lipo(X)$, and this proves that $\tau_p\subset\tau_{w^*}$. On the other hand, if $G$ is a finite subset of $\F(X)$, then $G$ is a norm bounded subset of $\Lipo(X)'$ and 
$p_G(f)\leq\sup_{\varphi\in G}\left\|\varphi\right\|\Lip(f)$ for all $f\in\Lipo(X)$, and this shows that $\tau_{w^*}\subset\tau_{\Lip}$. 

The ensuing result was proved by Godefroy and Kalton in \cite{gk}.

\begin{theorem}\label{mainth2}\cite{gk}
Let $E$ and $F$ be Banach spaces.
\begin{enumerate}
\item For every mapping $f\in\Lipo(E,F)$, there exists a unique operator $\widehat{f}\in\L(\F(E);\F(F))$ such that $\widehat{f}\circ\delta_E=\delta_F\circ f$. Furthermore, $||\widehat{f}||=\Lip(f)$.
\item If $E$ is a subspace of $F$ and $\iota\colon E\to F$ is the canonical embedding, then $\widehat{i}\colon\F(E)\to\F(F)$ is an isometric embedding.
\end{enumerate}
\end{theorem}

\section{The topology of bounded compact convergence for $\Lipo(X)$}\label{2}

We recall (see \cite[Definition 3.2]{c}) that a Saks space is a triple $(E,\left\|\cdot\right\|,\tau)$, where $E$ is a vector space, $\tau$ is a locally convex topology on $E$ and $\left\|\cdot\right\|$ is a norm on $E$ so that the closed unit ball $B_E$ of $(E,\left\|\cdot\right\|)$ is $\tau$-bounded and $\tau$-closed. 

Given a pointed metric space $X$, we consider on $\Lipo(X)$ the following topologies:
\begin{itemize}
	\item $\tau_p$ : the topology of pointwise convergence. 
	\item $\tau_0$ : the topology of compact convergence. 
	\item $\tau_{w^*}$ : the weak* topology $\sigma(\Lipo(X),\F(X))$. 
	\item $\tau_{\Lip}$ : the topology of the norm $\Lip$. 
\end{itemize}
 
The triple $(\Lipo(X),\Lip,\tau_0)$ is a Saks space since 
$B_{\Lipo(X)}$ is $\tau_0$-compact by the Ascoli theorem (see \cite[p. 234]{k}). 
Then, by \cite[3.4]{c}, we can form the mixed topology $\gamma[\Lip,\tau_0]$ on $\Lipo(X)$. Following \cite[Definition 1.4]{c}, $\gamma[\Lip,\tau_0]$ is the locally convex topology on $\Lipo(X)$ generated by the base of neighborhoods of zero $\left\{\gamma(U)\right\}$, where $U=\left\{U_n\right\}$ is a sequence of convex balanced $\tau_0$-neighborhoods of zero and 
$$
\gamma(U):=\bigcup_{n=1}^\infty\left(U_1\cap B_{\Lipo(X)}+U_2\cap 2B_{\Lipo(X)}+U_3\cap 2^2B_{\Lipo(X)}+\cdots+U_n\cap 2^{n-1}B_{\Lipo(X)}\right).
$$ 

Since $\tau_p\subset\tau_{w^*}\subset\tau_0$ on $B_{\Lipo(X)}$ (the second inclusion follows from Theorem \ref{mainth} (vi)) and $B_{\Lipo(X)}$ is $\tau_0$-compact, then $\tau_p=\tau_{w^*}=\tau_0$ on $B_{\Lipo(X)}$. Then \cite[Corollary 1.6]{c} yields   
$$
\gamma[\Lip,\tau_p]=\gamma[\Lip,\tau_{w^*}]=\gamma[\Lip,\tau_0],
$$
and we denote this topology by $\tau_{\gamma}$. We gather next some properties of $\tau_{\gamma}$. 

\begin{theorem}\label{tmf}
Let $X$ be a pointed metric space.
\begin{enumerate}
\item $\tau_0$ is smaller than $\tau_{\gamma}$, and $\tau_{\gamma}$ is smaller than $\tau_{\Lip}$.
\item $\tau_{\gamma}$ is the largest locally convex topology on $\Lipo(X)$ which coincides with $\tau_0$ on each norm bounded subset of $\Lipo(X)$.
\item If $F$ is a locally convex space and $T\colon\Lipo(X)\to F$ is linear, then $T$ is $\tau_{\gamma}$-continuous if and only if $\left.T\right|_B$ is $\tau_0$-continuous for each norm bounded subset $B$ of $\Lipo(X)$.
\item A sequence in $\Lipo(X)$ is $\tau_{\gamma}$-convergent to zero if and only if it is norm bounded and $\tau_0$-convergent to zero.
\item A subset of $\Lipo(X)$ is $\tau_{\gamma}$-bounded if and only if it is norm bounded.
\item A subset of $\Lipo(X)$ is $\tau_{\gamma}$-compact (precompact, relatively compact) if and only if it is norm bounded and $\tau_0$-compact (precompact, relatively compact).
\item $(\Lipo(X),\tau_{\gamma})$ is a complete semi-Montel space.	
\end{enumerate}
\end{theorem}

\begin{proof} 
The statements (i)--(vii) follow immediately from the theory of \cite[Chapter I]{c}. More namely, (i) and (ii) follow from Proposition 1.5; (iii) from Corollary 1.7; (iv) from Proposition 1.10; (v) from Proposition 1.11; (vi) from Proposition 1.12; and (vii) from Propositions 1.13 and 1.26 and the Ascoli theorem. 
\end{proof}

The property (ii) above justifies the name of topology of bounded compact convergence for $\tau_\gamma$. We next improve this property.

\begin{theorem}\label{new tmf}
Let $X$ be a pointed metric space.
\begin{enumerate}
\item $\tau_{\gamma}$ is the largest topology on $\Lipo(X)$ which agrees with $\tau_0$ on each norm bounded subset of $\Lipo(X)$.
\item A subset $U$ of $\Lipo(X)$ is open (closed) in $(\Lipo(X),\tau_{\gamma})$ if and only if $U\cap B$ is open (closed) in $(B,\tau_0)$ for each norm bounded subset $B$ of $\Lipo(X)$.
\item $\F(X)=(\Lipo(X),\tau_{\gamma})'_b=(\Lipo(X),\tau_{\gamma})'_c$.
\item The evaluation map $f\mapsto T_f$ from $(\Lipo(X),\tau_{\gamma})$ to $\F(X)'_c$ is a topological isomorphism.
\end{enumerate}
\end{theorem}

\begin{proof} 
The statements (i) and (ii) follow from \cite[Corollary 4.2]{c}. We now prove (iii). From Theorem \ref{mainth} (vi) and Theorem \ref{tmf} (ii)-(iii), we deduce that $\F(X)=(\Lipo(X),\tau_{\gamma})'$ algebraically. Since $\F(X)$ is a linear subspace of $(\Lipo(X),\tau_{\Lip})'$ by Theorem \ref{mainth} (i) and both spaces $(\Lipo(X),\tau_{\Lip})$ and $(\Lipo(X),\tau_{\gamma})$ have the same bounded sets by Theorem \ref{tmf} (v), we infer that $\F(X)=(\Lipo(X),\tau_{\gamma})'_b$. The identification $(\Lipo(X),\tau_{\gamma})'_b=(\Lipo(X),\tau_{\gamma})'_c$ follows from the fact that $(\Lipo(X),\tau_{\gamma})$ is a semi-Montel space, and the proof of (iii) is finished. 

To prove (iv), notice that 
$\left\{nB_{\Lipo(X)}\right\}$ is an increasing sequence of convex, balanced and $\tau_\gamma$-compact subsets of $\Lipo(X)$ (by Theorem \ref{tmf} (vi) and the Ascoli theorem) with the property that a set $U\subset\Lipo(X)$ is $\tau_\gamma$-open whenever $U\cap nB_{\Lipo(X)}$ is open in $(nB_{\Lipo(X)},\tau_\gamma)$ for every $n\in\N$. This property can be proved easily using the statements (i) and (ii).  
Then, by applying \cite[Theorem 4.1]{m}, the evaluation map from $(\Lipo(X),\tau_{\gamma})$ to $((\Lipo(X),\tau_{\gamma})'_b)'_c$ is a topological isomorphism. Since $\F(X)=(\Lipo(X),\tau_{\gamma})'_b$ by (iii), the statement (iv) holds.
\end{proof}

\begin{remark}\label{r1}
All assertions of Theorems \ref{tmf} and \ref{new tmf} are valid if the topology $\tau_0$ is replaced by $\tau_p$ or $\tau_{w^*}$. 
\end{remark}

We now recall that if $E$ is a Banach space, then the bounded weak* topology on its dual $E'$, denoted by $\tau_{bw^*}$, is the largest topology on $E'$ agreeing with the topology $\tau_{w^*}$ on norm bounded sets \cite[V.5.3]{ds}. According to the Banach-Dieudonn\'{e} theorem \cite[V.5.4]{ds}, $\tau_{bw^*}$ is just the topology of uniform convergence on sequences in $E$ which tend in norm to zero. 

Since $\tau_{w^*}=\tau_0$ on $B_{\Lipo(X)}$, the assertion (i) of Theorem \ref{new tmf} gives the following.

\begin{corollary}\label{tmf0}
Let $X$ be a pointed metric space. On the space $\Lipo(X)$, the bounded weak* topology $\tau_{bw^*}$ is the topology $\tau_\gamma$.
\end{corollary}

\section{Seminorm descriptions of $\tau_\gamma$ on $\Lipo(X)$}\label{3}

Our aim in this section is to give a pair of descriptions for $\tau_{\gamma}$ by means of seminorms. By Theorem \ref{new tmf} (ii) and Remark \ref{r1}, the convex balanced sets $U\subset\Lipo(X)$ such that $U\cap nB_{\Lipo(X)}$ is a neighborhood of zero in $(nB_{\Lipo(X)},\tau_{w^*})$ for every $n\in\N$, form a base of neighborhoods of zero for $\tau_{\gamma}$. For our purposes, we will need the next lemma. For $f\in\Lipo(X)$ and $A\subset X$, define   
$$
\Lip_A(f)=\sup\left\{\frac{\left|f(x)-f(y)\right|}{d(x,y)}\colon x,y\in A, \; x\neq y\right\}.
$$
Notice that if $F\subset X$ is finite, then $\Lip_F(f)=p_G(f)$ (see Section \ref{1}) where $G$ is the finite subset of $\F(X)$ given by 
$$
G=\left\{\frac{\delta_x-\delta_y}{d(x,y)}\colon x,y\in F, \; x\neq y\right\},
$$
and hence, for each $\varepsilon>0$, the set $\left\{f\in\Lipo(X)\colon\Lip_F(f)\leq\varepsilon\right\}$ is a neighborhood of $0$ in $(\Lipo(X),\tau_{w^*})$.

\begin{lemma}\label{l10}
Let $X$ be a pointed metric space. Then the sets of the form
$$
U=\bigcap_{n=1}^{\infty}\left\{f\in\Lipo(X)\colon\Lip_{F_n}(f)\leq\lambda_n\right\},
$$
where $\{F_n\}$ is a sequence of finite subsets of $X$ and $\{\lambda_n\}$ is a sequence of positive numbers tending to $\infty$, form a base of neighborhoods of zero in $(\Lipo(X),\tau_{\gamma})$.
\end{lemma}

\begin{proof}
We first claim that if $\{F_k\}$ and $\{\lambda_k\}$ are sequences as above, then the set 
$$
\bigcap_{k=1}^{\infty}\left\{f\in\Lipo(X)\colon\Lip_{F_k}(f)\leq\lambda_k\right\}
$$
is a neighborhood of $0$ in $(\Lipo(X),\tau_{\gamma})$. Indeed, given $n\in\N$, if $m\in\mathbb{N}$ is chosen so that $\lambda_k\geq n$ for $k>m$, then 
$$
\bigcap_{k=1}^{\infty}\left\{f\in\Lipo(X)\colon\Lip_{F_k}(f)\leq\lambda_k\right\}\cap n B_{\Lipo(X)}
=\bigcap_{n=1}^{m}\left\{f\in\Lipo(X)\colon\Lip_{F_k}(f)\leq\lambda_k\right\}\cap n B_{\Lipo(X)}.
$$
The latter is a neighborhood of $0$ in $(nB_{\Lipo(X)},\tau_{w^*})$, and this proves our claim.

We now must prove that if $U$ is a neighborhood of $0$ in $(\Lipo(X),\tau_{\gamma})$, then there are sequences $\{F_k\}$ and $\{\lambda_k\}$ as above for which 
$$
\bigcap_{k=1}^{\infty}\left\{f\in\Lipo(X)\colon\Lip_{F_k}(f)\leq\lambda_k\right\}\subset U.
$$
Indeed, we can take a set $U\subset\Lipo(X)$ such that $U\cap nB_{\Lipo(X)}$ is an open neighborhood of $0$ in $(nB_{\Lipo(X)},\tau_{w^*})$ for every $n\in\N$. In particular, $U\cap B_{\Lipo(X)}$ is a neighborhood of $0$ in $(B_{\Lipo(X)},\tau_{\Lip})$ and then there exists $\varepsilon>0$ such that $\varepsilon B_{\Lipo(X)}\subset U$. In order to prove that there exists a finite set $F_1\subset X$ such that 
$$
\left\{f\in\Lipo(X)\colon\Lip_{F_1}(f)\leq\varepsilon\right\}\cap B_{\Lipo(X)}\subset U,
$$
assume on the contrary that the set
$$
\left\{f\in\Lipo(X)\colon\Lip_{F}(f)\leq\varepsilon\right\}\cap(B_{\Lipo(X)}\setminus U)
$$
is nonempty for every finite set $F\subset X$. These sets are closed in $(B_{\Lipo(X)}\setminus U,\tau_{w^*})$ and have the finite intersection property. Since the set
$B_{\Lipo(X)}\setminus U$ is a closed, and therefore compact, subset of $(B_{\Lipo(X)},\tau_{w^*})$, we infer that there exists some $f\in B_{\Lipo(X)}\setminus U$ such that $\Lip_{F}(f)\leq\varepsilon$ for each finite set $F\subset X$. This implies that $f\in\varepsilon B_{\Lipo(X)}\setminus U$ which is impossible, and thus proving our assertion. 

Proceeding by induction, suppose that we can find finite subsets $F_2,\ldots,F_n$ of $X$ such that 
$$
\bigcap_{k=1}^n\left\{f\in\Lipo(X)\colon\Lip_{F_k}(f)\leq\lambda_k\right\}\cap nB_{\Lipo(X)}\subset U\cap nB_{\Lipo(X)},
$$
where $\lambda_1=\varepsilon$ and $\lambda_k=k-1$ for $k>1$. We will prove that there exists a finite set $F_{n+1}\subset X$ such that
$$
\bigcap_{k=1}^{n+1}\left\{f\in\Lipo(X)\colon\Lip_{F_k}(f)\leq\lambda_k\right\}\cap (n+1)B_{\Lipo(X)}\subset U\cap (n+1)B_{\Lipo(X)}.
$$
We argue by contradiction. If no such finite set  $F_{n+1}$ exists, then the set
$$
C_F:=\bigcap_{k=1}^n\left\{f\in\Lipo(X)\colon\Lip_{F_k}(f)\leq\lambda_k\right\}\cap\left\{f\in\Lipo(X)\colon\Lip_F(f)\leq n\right\}
$$
has nonempty intersection with the $\tau_{w^*}$-compact set $(n+1)B_{\Lipo(X)}\setminus U$ for each finite set $F\subset X$. So, by the finite intersection property, there is a $f_0\in \left((n+1)B_{\Lipo(X)}\setminus U\right)\cap\left(\cap_{F}C_F\right)$. Therefore $\Lip_F(f_0)\leq n$ for each $F$ and so $\Lip(f_0)\leq n$. Then $f_0\in U\cap nB_{\Lipo(X)}\subset U\cap (n+1)B_{\Lipo(X)}$ which is a contradiction.

Then we can construct, by induction, a sequence $\{F_k\}$ of finite subsets of $X$ so that
$$
\bigcap_{k=1}^{n}\left\{f\in\Lipo(X)\colon\Lip_{F_k}(f)\leq\lambda_k\right\}\cap nB_{\Lipo(X)}\subset U
$$
for every $n\in\N$. Since $\Lipo(X)=\cup_{n=1}^\infty nB_{\Lipo(X)}$, we conclude that 
$$
\bigcap_{k=1}^{\infty}\left\{f\in\Lipo(X)\colon\Lip_{F_k}(f)\leq\lambda_k\right\}\subset U.
$$ 
\end{proof}

A first characterization of $\tau_{\gamma}$ by means of seminorms lies over the concept of strict topology, introduced by Buck in \cite{bu}, for spaces of continuous functions on locally compact spaces.

Let $X$ be a pointed metric space. We denote
$$
\widetilde{X}=\left\{(x,y)\in X^2\colon x\neq y\right\}.
$$
Let $C_b(\widetilde{X})$ be the space of bounded continuous scalar-valued functions on $\widetilde{X}$ with the supremum norm, and let $\Phi$ be De Leeuw's map from $\Lipo(X)$ into $C_b(\widetilde{X})$ defined by 
$$
\Phi(f)(x,y)=\frac{f(x)-f(y)}{d(x,y)}.
$$
Clearly, $\Phi$ is an isometric isomorphism from $\Lipo(X)$ onto the closed subspace $\Phi(\Lipo(X))$ 
of $C_b(\widetilde{X})$.

\begin{definition}
Let $X$ be a compact pointed metric space. The strict topology $\beta$ on $\Lipo(X)$ is the strict topology on $\Phi(\Lipo(X))$, that is the locally convex topology generated by the seminorms of the form 
$$
\left\|f\right\|_\phi=\sup_{(x,y)\in\widetilde{X}}\left|\phi(x,y)\right|\frac{\left|f(x)-f(y)\right|}{d(x,y)},\qquad f\in\Lipo(X),
$$
where $\phi$ runs through the space $C_0(\widetilde{X})$ of continuous functions from $\widetilde{X}$ into $\K$ which vanish at infinity. 
\end{definition}

\begin{theorem}\label{sha}
Let $X$ be a compact pointed metric. On the space $\Lipo(X)$, the strict topology $\beta$ is the topology $\tau_\gamma$.
\end{theorem}

\begin{proof}
We first show that the identity is a continuous mapping from $(\Lipo(X),\tau_\gamma)$ to $(\Lipo(X),\beta)$. By Theorem \ref{tmf} (iii), it is enough to show that the identity on $nB_{\Lipo(X)}$ is continuous on $(nB_{\Lipo(X)},\tau_0)$ for every $n\in\N$. Let $n\in\N$ and fix $\phi\in C_0(\widetilde{X})$ and $\varepsilon>0$. Then there is a compact set $K\subset\widetilde{X}$ such that $\left|\phi(x,y)\right|<\varepsilon/2n$ if $(x,y)\in \widetilde{X}\setminus K$. Take 
$$
U=\left\{f\in\Lipo(X)\colon \sup_{(x,y)\in K}\frac{\left|f(x)-f(y)\right|}{d(x,y)}\leq\frac{\varepsilon}{2(1+\left\|\phi\right\|_\infty)}\right\}.
$$ 
We now prove that $U$ is a neighborhood of $0$ in $(\Lipo(X),\tau_\gamma)$. Indeed, define $\sigma\colon\widetilde{X}\to\F(X)$ by 
$$
\sigma(x,y)=\frac{\delta_x-\delta_y}{d(x,y)}.
$$
Since the maps $x\mapsto\delta_x$ and $(x,y)\mapsto d(x,y)$ are continuous, so is also $\sigma$. Then $\sigma(K)$ is a compact subset of $\F(X)$ and therefore the polar 
$$
\sigma(K)^\circ:=\left\{F\in\F(X)'\colon \sup_{(x,y)\in K}\left|F(\sigma(x,y))\right|\leq 1\right\} 
$$
is a neighborhood of $0$ in $\F(X)'_c$. Then, by Theorem \ref{new tmf} (iv), the set 
$\left\{f\in\Lipo(X)\colon T_f\in\sigma(K)^\circ\right\}$, that is
$$
\left\{f\in\Lipo(X)\colon \sup_{(x,y)\in K}\frac{\left|f(x)-f(y)\right|}{d(x,y)}\leq 1\right\},
$$
is a neighborhood of $0$ in $(\Lipo(X),\tau_\gamma)$, and hence so is $U$ as required. It follows that $U\cap nB_{\Lipo(X)}$ is a neighborhood of $0$ in $(nB_{\Lipo(X)},\tau_0)$ by Theorem \ref{new tmf} (ii). If $f\in U\cap n B_{\Lipo(X)}$, we have 
\begin{align*}
\left\|f\right\|_\phi
&\leq\sup_{(x,y)\in K}\left|\phi(x,y)\right|\frac{\left|f(x)-f(y)\right|}{d(x,y)}
+\sup_{(x,y)\in\widetilde{X}\setminus K}\left|\phi(x,y)\right|\frac{\left|f(x)-f(y)\right|}{d(x,y)}\\
&\leq\left\|\phi\right\|_\infty\frac{\varepsilon}{2(1+\left\|\phi\right\|_\infty)}+\frac{\varepsilon}{2n}n<\varepsilon .
\end{align*}
Conversely, let $U$ be a neighborhood of $0$ in $(\Lipo(X),\tau_\gamma)$. By Lemma \ref{l10}, we can suppose that 
$$
U=\bigcap_{n=1}^{\infty}\left\{f\in\Lipo(X)\colon\Lip_{F_n}(f)\leq\lambda_n\right\}
$$
where $\{F_n\}$ is a sequence of finite subsets of $X$ and $\{\lambda_n\}$ is a sequence of positive numbers tending to $\infty$. We can further suppose that $F_n\subset F_{n+1}$ and $\lambda_{n}<\lambda_{n+1}$ for all $n\in\N$. We can construct a function $\phi$ in $C_0(\widetilde{X})$ with $\{(x,y)\in\widetilde{X}\colon \phi(x,y)\neq 0\}\subset\cup_{n=1}^\infty F_n$ so that $\phi(x,y)=1/\lambda_1$ if $(x,y)\in F_1$ and $1/\lambda_{n+1}\leq\phi(x,y)\leq 1/\lambda_n$ for all $(x,y)\in F_{n+1}\setminus F_n$. Then $\{f\in\Lipo(X)\colon \left\|f\right\|_\phi\leq 1\}\subset U$ and this proves the theorem.
\end{proof}



The second description of $\tau_{\gamma}$ in terms of seminorms is the ensuing.

\begin{theorem}\label{c9}
Let $X$ be a pointed metric space. The topology $\tau_{\gamma}$ is generated by the seminorms of the form
$$
p(f)=\sup_{n\in\N}\alpha_n\frac{\left|f(x_n)-f(y_n)\right|}{d(x_n,y_n)},\qquad f\in\Lipo(X),
$$
where $\{\alpha_n\}$ varies over all sequences in $\R^+$ tending to zero and $\{(x_n,y_n)\}$ runs over all sequences in $\widetilde{X}$.
\end{theorem}

\begin{proof}
Let $\mathcal{V}$ be the base of neighborhoods of $0$ in $(\Lipo(X),\tau_{\gamma})$ formed by the sets of the form
$$
U=\bigcap_{n=1}^{\infty}\left\{f\in\Lipo(X)\colon\Lip_{F_n}(f)\leq\lambda_n\right\}
$$
where $\{F_n\}$ and $\{\lambda_n\}$ are sequences as in Lemma \ref{l10}. If, for each $U\in\mathcal{V}$, $p_U$ is the Minkowski functional of $U$, then the family of seminorms $\left\{p_U\colon U\in\mathcal{V}\right\}$ generates the topology $\tau_{\gamma}$ on $\Lipo(X)$, but justly we have 
$$
p_U(f)=\sup_{n\in\N}\lambda_n^{-1}\Lip_{F_n}(f)
$$
for all $f\in\Lipo(X)$, and the result follows. 
\end{proof}



Let $E$ be a Banach space. The polar of $M\subset E$ and the prepolar of $N\subset E'$ are respectively
\begin{align*}
M^{\circ}&=\left\{f\in E'\colon\sup_{x\in M}\left|f(x)\right|\leq 1\right\},\\
N_{\circ}&=\left\{x\in E\colon\sup_{f\in N}\left|f(x)\right|\leq 1\right\}.
\end{align*}
$\overline{\Gamma}M$ stands for the closed, convex, balanced hull of $M$ in $E$. The next lemma will be needed later. 

\begin{lemma}\label{c0}
Let $X$ be a pointed metric space. For each compact set $L\subset\F(X)$, there exist sequences $\{\alpha_n\}\in c_0(\R^+)$ and $\{(x_n,y_n)\}\in \widetilde{X}^{\N}$ such that
$$
L\subset\overline{\Gamma}\left\{\alpha_n\frac{\delta_{x_n}-\delta_{y_n}}{d(x_n,y_n)}\colon n\in\N\right\}.
$$
\end{lemma}

\begin{proof}
If $L$ is a compact subset of $\F(X)$, then the polar $L^0:=\{F\in\F(X)'\colon\sup_{\varphi\in L}|F(\varphi)|\leq 1\}$ is a neighborhood of $0$ in $\F(X)_c'$. Then, by Theorem \ref{new tmf} (iv), the set $\left\{f\in\Lipo(X)\colon T_f\in L^0\right\}$ is a neighborhood of $0$ in $(\Lipo(X),\tau_{\gamma})$. Hence, by Theorem \ref{c9}, there exist sequences $\{\alpha_n\}\in c_0(\R^+)$ and $\{(x_n,y_n)\}\in \widetilde{X}^{\N}$ such that 
$$
\left\{f\in\Lipo(X)\colon\sup_{n\in\N}\alpha_n\frac{\left|f(x_n)-f(y_n)\right|}{d(x_n,y_n)}\leq 1\right\}
\subset\left\{f\in\Lipo(X)\colon\sup_{\varphi\in L}\left|T_f(\varphi)\right|\leq 1\right\}.
$$ 
We have
\begin{align*}
\left\{\alpha_n\frac{\delta_{x_n}-\delta_{y_n}}{d(x_n,y_n)}\colon n\in\N\right\}^{0}
&=\left\{F\in\F(X)'\colon\sup_{n\in N}\left|F\left(\alpha_n\frac{\delta_{x_n}-\delta_{y_n}}{d(x_n,y_n)}\right)\right|\leq 1\right\}\\
&=\left\{T_f\colon f\in\Lipo(X),\, \sup_{n\in N}\alpha_n\frac{\left|f(x_n)-f(y_n)\right|}{d(x_n,y_n)}\leq 1\right\}\\
&\subset\left\{T_f\colon f\in\Lipo(X),\, \sup_{\varphi\in L}\left|T_f(\varphi)\right|\leq 1\right\}\\
&=\left\{F\in\F(X)'\colon \sup_{\varphi\in L}\left|F(\varphi)\right|\leq 1\right\}\\
&=L^{0},
\end{align*}
and then the bipolar theorem yields 
$$
L
\subset(L^{0})_0
\subset\left(\left\{\alpha_n\frac{\delta_{x_n}-\delta_{y_n}}{d(x_n,y_n)}\colon n\in\N\right\}^{0}\right)_0
=\overline{\Gamma}\left\{\alpha_n\frac{\delta_{x_n}-\delta_{y_n}}{d(x_n,y_n)}\colon n\in\N\right\}.
$$
\end{proof}

\section{The approximation property for $(\Lipo(X),\tau_\gamma)$}\label{4}

We devote this section to the study of the (AP) for the space $\Lipo(X)$ with the topology of bounded compact convergence. For it, we introduce the subsequent topology on $\Lipo(X,F)$.

\begin{definition}\label{defini}
Let $X$ be a pointed metric and let $F$ be a Banach space. The topology $\gamma\tau_{\gamma}$ on $\Lipo(X,F)$ is the locally convex topology generated by the seminorms of the form
$$
q(f)=\sup_{n\in\N}\alpha_n\frac{\left\|f(x_n)-f(y_n)\right\|}{d(x_n,y_n)},\qquad f\in\Lipo(X,F),
$$
where $\{\alpha_n\}$ ranges over the sequences in $\R^+$ tending to zero and $\{(x_n,y_n)\}$ over the sequences in $\widetilde{X}$.
\end{definition}

We study the relation between the topologies $\gamma\tau_{\gamma}$, $\tau_{\gamma}$ and $\tau_0$.

\begin{proposition}\label{p0}
Let $X$ be a pointed metric space and let $F$ be a Banach space. 
\begin{enumerate}
	\item $\tau_0$ is smaller than $\gamma\tau_{\gamma}$ on $\Lipo(X,F)$.
	\item $\tau_\gamma$ agrees with $\gamma\tau_{\gamma}$ on $\Lipo(X)$.
\end{enumerate}
\end{proposition}

\begin{proof}
To prove (i), let $K$ be a compact subset of $X$. Then $\delta_X(K)$ is a compact subset of $\F(X)$ and, by Lemma \ref{c0}, there are sequences $\{\alpha_n\}\in c_0(\R^+)$ and $\{(x_n,y_n)\}\in \widetilde{X}^{\N}$ such that
$$
\delta_X(K)\subset\overline{\Gamma}\left\{\alpha_n\frac{\delta_{x_n}-\delta_{y_n}}{d(x_n,y_n)}\colon n\in\N\right\}.
$$
It follows that   
$$
\left|f\right|_K
=\sup_{x\in K}\left\|f(x)\right\|
\leq\sup_{n\in\N}\alpha_n\frac{\left\|f(x_n)-f(y_n)\right\|}{d(x_n,y_n)}
=q(f),
$$
for all $f\in\Lipo(X,F)$, as desired. (ii) is deduced from Theorem \ref{c9} and Definition \ref{defini}.
\end{proof}

If $f\in\Lipo(X,F)$, we define the Lipschitz transpose of $f$ to the linear mapping $f^t\colon F'\to\Lipo(X)$ given by $f^t(\psi)=\psi\circ f$ for all $\psi\in F'$. Our next result shows that the Lipschitz transpose can be used to identify the space $(\Lipo(X,F),\gamma\tau_\gamma)$ with $(\Lipo(X),\tau_\gamma)\epsilon F$. By Section \ref{1}, notice that the seminorms
$$
\sup\left\{\alpha_n\frac{\left|T(\psi)(x_n)-T(\psi)(y_n)\right|}{d(x_n,y_n)}\colon n\in\N,\; \psi\in F',\;\left\|\psi\right\|\leq 1\right\},
\qquad T\in (\Lipo(X),\tau_\gamma)\epsilon F,
$$
where $\{\alpha_n\}$ and $\{(x_n,y_n)\}$ are sequences as above, determine the topology of $\L_\epsilon(F'_c;(\Lipo(X),\tau_\gamma))=(\Lipo(X),\tau_\gamma)\epsilon F$.

\begin{theorem}\label{t1-211}
Let $X$ be a pointed metric space and let $F$ be a Banach space. The mapping $f\mapsto f^t$ is a topological isomorphism from $(\Lipo(X,F),\gamma\tau_\gamma)$ onto $(\Lipo(X),\tau_\gamma)\epsilon F$. 
\end{theorem}

\begin{proof}
If $f\in\Lipo(X,F)$, the mapping $f^t\colon F'\to\Lipo(X)$ is continuous from $F'_c$ into $(\Lipo(X),\tau_\gamma)$. To prove this, let $p$ be a continuous seminorm on $(\Lipo(X),\tau_\gamma)$. By Theorem \ref{c9}, we can suppose that 
$$
p(g)=\sup_{n\in\N}\alpha_n\frac{\left|g(x_n)-g(y_n)\right|}{d(x_n,y_n)},\qquad g\in\Lipo(X),
$$
where $\{\alpha_n\}$ is a sequence in $\R^+$ tending to zero and $\{(x_n,y_n)\}$ is a sequence in $\widetilde{X}$. Since 
$$
\left\|\alpha_n\frac{f(x_n)-f(y_n)}{d(x_n,y_n)}\right\|\leq\alpha_n\Lip(f)
$$
for all $n\in\N$, the set 
$$
K=\left\{\alpha_n\frac{f(x_n)-f(y_n)}{d(x_n,y_n)}\right\}\cup\left\{0\right\}
$$
is compact in $F$. For each $\psi\in F'$, we have
$$
\alpha_n\frac{\left|f^t(\psi)(x_n)-f^t(\psi)(y_n)\right|}{d(x_n,y_n)}
=\left|\psi\left(\alpha_n\frac{f(x_n)-f(y_n)}{d(x_n,y_n)}\right)\right|
\leq\left|\psi\right|_K
$$
for all $n\in\N$, and consequently $p(f^t(\psi))\leq\left|\psi\right|_K$ as required.

Clearly, the mapping $f\mapsto f^t$ from $\Lipo(X,F)$ to $\L_\epsilon(F'_c;(\Lipo(X),\tau_\gamma))$ is linear and, since $F'$ separates the points of $F$, is injective. To prove that it is surjective, let $T\in\L_\epsilon(F'_c;(\Lipo(X),\tau_\gamma))$. Then its transpose $T^t$ is in $\L_\epsilon((\Lipo(X),\tau_\gamma)'_c;F)=\L_\epsilon(\F(X);F)$ by Proposition \ref{r0} and Theorem \ref{new tmf} (iii). Notice that $T\in\L(F';\Lipo(X))$ since the closed unit ball $B_{F'}$ of $F'$ is a compact subset of $(F',\tau_0)$, then $T(B_{F'})$ is a bounded subset of $(\Lipo(X),\tau_\gamma)$ and hence norm bounded by Theorem \ref{tmf} (v). Consider now the Dirac map $\delta_X\colon X\to\F(X)$. Then the mapping $f=T^t\circ\delta_X$ maps $X$ into $F$, vanishes at $0$ and is Lipschitz since  
$$
\left\|f(x)-f(y)\right\|
\leq\left\|T^t\right\|\left\|\delta_X(x)-\delta_X(y)\right\|
=\left\|T\right\|d(x,y)
$$
for all $x,y\in X$. For every $\psi\in F'$ and $x\in X$, we have
$$
f^t(\psi)(x)=\left\langle\psi,f(x)\right\rangle=\left\langle\psi,T^t\delta_X(x)\right\rangle=\left\langle T(\psi),\delta_X(x)\right\rangle=T(\psi)(x),
$$
and thus $f^t=T$. Hence the mapping $f\mapsto f^t$ is a linear bijection from $\Lipo(X,F)$ onto $(\Lipo(X),\tau_\gamma)\epsilon F$ with inverse given by $T\mapsto T^t\circ\delta_X$. 

It remains tho show that it is continuous with continuous inverse. For it, let $\{\alpha_n\}$ and $\{(x_n,y_n)\}$ be sequences as above. By Definition \ref{defini},  
$$
q(f)=\sup_{n\in\N}\alpha_n\frac{\left\|f(x_n)-f(y_n)\right\|}{d(x_n,y_n)},\qquad f\in\Lipo(X,F),
$$
is a continuous seminorm on $(\Lipo(X,F),\gamma\tau_\gamma)$. If $n\in\N$ and $\psi\in F'$ with $\left\|\psi\right\|\leq 1$, we have  
$$
\alpha_n\frac{\left|f^t(\psi)(x_n)-f^t(\psi)(y_n)\right|}{d(x_n,y_n)}
=\alpha_n\frac{\left|\psi(f(x_n))-\psi(f(y_n))\right|}{d(x_n,y_n)}
\leq\alpha_n\frac{\left\|f(x_n)-f(y_n)\right\|}{d(x_n,y_n)},
$$
therefore 
$$
\sup\left\{\alpha_n\frac{\left|f^t(\psi)(x_n)-f^t(\psi)(y_n)\right|}{d(x_n,y_n)}\colon n\in\N,\;\psi\in F',\;\left\|\psi\right\|\leq 1\right\}\leq q(f)
$$
and this proves that the mapping $f\mapsto f^t$ is continuous. To see that its inverse $T\mapsto T^t\circ\delta_X$ is continuous, 
let $q$ be a continuous seminorm on $(\Lipo(X,F),\gamma\tau_\gamma)$. By definition \ref{defini}, we can suppose that 
$$
q(f)=\sup_{n\in\N}\alpha_n\frac{\left\|f(x_n)-f(y_n)\right\|}{d(x_n,y_n)},\qquad f\in\Lipo(X,F),
$$
where $\{\alpha_n\}$ and $\{(x_n,y_n)\}$ are sequences as above. For each $n\in\N$, take $\psi_n\in B_{F'}$ such that 
$$
\left\|T^t\delta_X(x_n)-T^t\delta_X(y_n)\right\|
=\left|\left\langle\psi_n,T^t\delta_X(x_n)-T^t\delta_X(y_n)\right\rangle\right|
$$
and then we have 
$$
\alpha_n\frac{\left\|T^t\delta_X(x_n)-T^t\delta_X(y_n)\right\|}{d(x_n,y_n)}
=\alpha_n\frac{\left|\left\langle T\psi_n,\delta_X(x_n)-\delta_X(y_n)\right\rangle\right|}{d(x_n,y_n)}
=\alpha_n\frac{\left|T(\psi_n)(x_n)-T(\psi_0)(y_n)\right|}{d(x_n,y_n)}.
$$
It follows that   
$$
q(T^t\circ\delta_X)\leq
\sup\left\{\alpha_n\frac{\left|T(\psi)(x_n)-T(\psi)(y_n)\right|}{d(x_n,y_n)}\colon n\in\N,\;\psi\in F',\;\left\|\psi\right\|\leq 1\right\},
$$
and the proof is finished.
\end{proof}

Our next aim is to identify linearly the tensor product $\Lipo(X)\otimes F$ with the space of all Lipschitz finite-rank operators from $X$ to $F$. 
Let us recall that a mapping $f\in\Lipo(X,F)$ is called a Lipschitz finite-rank operator if the linear hull of $f(X)$ in $F$ has finite dimension in whose case this dimension is called the rank of $f$ and denoted by $\rank(f)$. We represent by $\LipoF(X,F)$ the vector space of all Lipschitz finite-rank operators from $X$ to $F$. This space can be generated linearly as follows.

\begin{lemma}\label{lem-finite}
Let $X$ be a pointed metric space and $F$ a Banach space.
\begin{enumerate}
\item If $g\in \Lipo(X)$ and $u\in F$, then the map $g\cdot u\colon X\to F$, given by $(g\cdot u)(x)=g(x)u$, belongs to $\LipoF(X,F)$ and $\Lip(g\cdot u)=\Lip(g)\left\|u\right\|$. Moreover, $\rank(g\cdot u)=1$ if $g\neq 0$ and $ u\neq 0$.
\item Every element $f\in\LipoF(X,F)$ has a representation in the form $f=\sum_{j=1}^m g_j\cdot  u_j$, where $m=\rank(f)$, $g_1,\ldots,g_m\in\Lipo(X)$ and $ u_1,\ldots, u_m\in F$.
\end{enumerate}
\end{lemma}

\begin{proof}
(i) Clearly, $g\cdot u$ is well-defined. Let $x,y\in X$. For any $u\in F$, we obtain 
$$
\left\|(g\cdot u)(x)-(g\cdot u)(y)\right\|
=\left\|(g(x)-g(y))u\right\|
=\left|g(x)-g(y)\right|\left\|u\right\|
\leq\Lip(g)d(x,y)\left\| u\right\|,
$$
and so $g\cdot u\in\Lipo(X,F)$ and $\Lip(g\cdot u)\leq\Lip(g)\left\|u\right\|$. For the converse inequality, note that 
$$
\left|g(x)-g(y)\right|\left\| u\right\|
=\left\|(g\cdot u)(x)-(g\cdot u)(y)\right\|
\leq\Lip(g\cdot u)d(x,y)
$$ 
for all $x,y\in X$, and therefore $\Lip(g)\left\| u\right\|\leq\Lip(g\cdot u)$.

(ii) Suppose that the linear hull $\lin(f(X))$ of $f(X)$ in $F$ is $m$-dimensional and let $\{ u_1,\ldots, u_m\}$ be a base of $\lin(f(X))$. Then, for each $x\in X$, the element $f(x)\in f(X)$ is expressible in a unique form as $f(x)=\sum_{j=1}^m \lambda^{(x)}_j u_j$ with $\lambda^{(x)}_1,\ldots,\lambda^{(x)}_m\in\K$. For each $j\in\{1,\ldots,m\}$, define the linear map $y^j\colon\lin(f(X))\to\K$ by $y^j(f(x))=\lambda^{(x)}_j$ for all $x\in X$. Let $g_j=y ^j\circ f$. Clearly, $g_j\in\Lipo(X)$ and $f(x)=\sum_{j=1}^m \lambda^{(x)}_j u_j=\sum_{j=1}^m g_j(x) u_j$ for all $x\in X$. Hence $f=\sum_{j=1}^m g_j\cdot u_j$.
\end{proof}


\begin{proposition}\label{p-identification2}
Let $X$ be a pointed metric space and let $F$ be a Banach space. Then $\Lipo(X)\otimes F$ is linearly isomorphic to $\LipoF(X,F)$ via the linear bijection $K\colon \Lipo(X)\otimes F\to\LipoF(X,F)$ given by 
$$
K\left(\sum_{j=1}^m g_j\otimes u_j\right)=\sum_{j=1}^m g_j\cdot u_j.
$$
\end{proposition}

\begin{proof}
Let $\sum_{j=1}^m g_j\otimes u_j\in \Lipo(X)\otimes F$. The mapping $K$ is well defined. Indeed, if $\sum_{j=1}^m g_j\otimes u_j=0$, then $\sum_{j=1}^m\varphi(g_j)u_j=0$ for every $\varphi\in\Lipo(X)'$ by \cite[Proposition 1.2]{ryan}. In particular, $\sum_{j=1}^m \delta_x(g_j)u_j=0$ for every $x\in X$ and thus $\sum_{j=1}^m g_j\cdot u_j=0$ as required. Clearly, $K$ is linear and, by Lemma \ref{lem-finite}, is onto. To see that it is injective, assume that $K(\sum_{j=1}^m g_j\otimes u_j)=0$. Then $\sum_{j=1}^m \delta_x(g_j)u_j=0$ for every $x\in X$, and since $\left\{\delta_x\colon x\in X\right\}$ is a separating subset of $\Lipo(X)'$, we infer that $\sum_{j=1}^m g_j\otimes u_j=0$ (see \cite[p. 3-4]{ryan}).
\end{proof}

From the preceding results we deduce the next result that characterizes the (AP) for the space $\Lipo(X)$ with the topology of bounded compact convergence. 
 
\begin{corollary}\label{c00000}
Let $X$ be a pointed metric space. The following are equivalent.
\begin{enumerate}
	\item $(\Lipo(X),\tau_\gamma)$ has the (AP).
	\item $\F(X)$ has the (AP).
	\item $\Lipo(X)\otimes F$ is dense in $(\Lipo(X),\tau_\gamma)\epsilon F$ for every Banach space $F$.
	\item $\LipoF(X,F)$ is dense in $(\Lipo(X,F),\gamma\tau_\gamma)$ for every Banach space $F$.
\end{enumerate}
\end{corollary}

\begin{proof}
(i) $\Leftrightarrow$ (ii): Assume that $(\Lipo(X),\tau_\gamma)$ has the (AP). Since $(\Lipo(X),\tau_\gamma)=F(X)'_c$ by Theorem \ref{new tmf} (iv), then $\F(X)$ has the (AP) by Proposition \ref{r000}. Conversely, if $\F(X)$ has the (AP), we use that $\F(X)=(\Lipo(X),\tau_\gamma)'_c$ by Theorem \ref{new tmf} (iii) to obtain that $(\Lipo(X),\tau_\gamma)$ has the (AP) by Proposition \ref{r000}.

(i) $\Leftrightarrow$ (iii) is an application of Theorem \ref{r00}, and (iii) $\Leftrightarrow$ (iv) follows from Proposition \ref{p-identification2} and Theorem \ref{t1-211}. 
\end{proof}

\section{The dual space of $(\Lipo(X,F),\gamma\tau_{\gamma})$}\label{5}

The following theorem describes the dual of the space $(\Lipo(X,F),\gamma\tau_{\gamma})$. Recall that a linear functional $T$ on a topological vector space $Y$ is continuous if and only if there is a neighborhood $U$ of $0$ in $Y$ such that $T(U)$ is a bounded subset of $\K$. 
Hence $T\in(\Lipo(X,F),\gamma\tau_{\gamma})'$ if and only if there exist a constant $c>0$ and sequences $\{\alpha_n\}\in c_0(\R^+)$ and $\{(x_n,y_n)\}\in \widetilde{X}^\N$ such that 
$$
\left|T(f)\right|\leq c\sup_{n\in\N}\alpha_n\frac{\left\|f(x_n)-f(y_n)\right\|}{d(x_n,y_n)}
$$
for every $f\in\Lipo(X,F)$.

\begin{theorem}\label{iso22}
Let $X$ be a pointed metric and let $F$ be a Banach space. Then a linear functional $T$ on $\Lipo(X,F)$ is in the dual of $(\Lipo(X,F),\gamma\tau_{\gamma})$ if and only if there exist sequences $\{\phi_n\}$ in $F'$ and $\{(x_n,y_n)\}$ in $\widetilde{X}$ such that $\sum_{n=1}^\infty\left\|\phi_n\right\|<\infty$ and 
$$
T(f)=\sum_{n=1}^\infty\phi_n\left(\frac{f(x_n)-f(y_n)}{d(x_n,y_n)}\right)
$$
for all $f\in\Lipo(X,F)$. 
\end{theorem}

\begin{proof}
Assume that $T$ is a linear functional on $\Lipo(X,F)$ of the preceding form. 
Since $\sum_{n=1}^\infty\left\|\phi_n\right\|<\infty$, we can take a sequence $\{\lambda_n\}$ in $\R^+$ tending to $\infty$ so that $\sum_{n=1}^\infty\lambda_n\left\|\phi_n\right\|=c<\infty$. Then we have
$$
\left|T(f)\right|
\leq\sum_{n=1}^\infty\left\|\phi_n\right\|\frac{\left\|f(x_n)-f(y_n)\right\|}{d(x_n,y_n)}
\leq c\sup_{n\in\N}\lambda_n^{-1}\frac{\left\|f(x_n)-f(y_n)\right\|}{d(x_n,y_n)}
$$
for all $f\in\Lipo(X,F)$. This proves that $T$ is continuous on $(\Lipo(X,F),\gamma\tau_{\gamma})$.

Conversely, if $T\in(\Lipo(X,F),\gamma\tau_{\gamma})'$, then there are sequences $\{\alpha_n\}\in c_0(\R^+)$ and $\{(x_n,y_n)\}\in\widetilde{X}^\N$ such that
$$
\left|T(f)\right|\leq\sup_{n\in\N}\alpha_n\frac{\left\|f(x_n)-f(y_n)\right\|}{d(x_n,y_n)}
$$
for every $f\in\Lipo(X,F)$. Consider the linear subspace 
$$
Z=\left\{\left\{\alpha_n\frac{f(x_n)-f(y_n)}{d(x_n,y_n)}\right\}\colon f\in\Lipo(X,F)\right\}
$$
of $c_0(F)$, and the functional $S$ on $Z$ given by 
$$
S\left(\left\{\alpha_n\frac{f(x_n)-f(y_n)}{d(x_n,y_n)}\right\}\right)=T(f)
$$
for every $f\in\Lipo(X,F)$. It follows easily that $S$ is well defined and linear. 
Since 
$$
\left|S\left(\left\{\alpha_n\frac{f(x_n)-f(y_n)}{d(x_n,y_n)}\right\}\right)\right|
=\left|T(f)\right|
\leq\sup_{n\in\N}\alpha_n\frac{\left\|f(x_n)-f(y_n)\right\|}{d(x_n,y_n)}
=\left\|\left\{\alpha_n\frac{f(x_n)-f(y_n)}{d(x_n,y_n)}\right\}\right\|_\infty
$$
for all $f\in\Lipo(X,F)$, $S$ is continuous on $Z$. By the Hahn-Banach theorem, $S$ has a norm-preserving continuous linear extension $\widehat{S}$ to all of $c_0(F)$. Since $c_0(F)'$ is just $\ell_1(F')$, there exists a sequence $\{\psi_n\}$ in $F'$ such that $\sum_{n=1}^\infty\left\|\psi_n\right\|=||\widehat{S}||$ and $\widehat{S}(\{u_n\})=\sum_{n=1}^\infty\psi_n(u_n)$ for any $\{u_n\}\in c_0(F)$. Taking $\phi_n=\alpha_n\psi_n$ for each $n\in\N$, we conclude that $\sum_{n=1}^\infty\left\|\phi_n\right\|\leq\left\|\{\alpha_n\}\right\|_\infty||\widehat{S}||<\infty$ and 
$$
T(f)
=\widehat{S}\left(\left\{\alpha_n\frac{f(x_n)-f(y_n)}{d(x_n,y_n)}\right\}\right)
=\sum_{n=1}^\infty\phi_n\left(\frac{f(x_n)-f(y_n)}{d(x_n,y_n)}\right)
$$
for all $f\in\Lipo(X,F)$.
\end{proof}

Since $\F(X)=(\Lipo(X),\tau_{\gamma})'_b$ by Theorem \ref{new tmf} (iii) and $\tau_{\gamma}=\gamma\tau_{\gamma}$ on $\Lipo(X)$ by Proposition \ref{p0} (ii), we next apply Theorem \ref{iso22} to describe the members of $\F(X)$. 

\begin{corollary}
Let $X$ be a pointed metric. Then $\F(X)$ consists of all functionals $T\in\Lipo(X)'$ of the form
$$
T(f)=\sum_{n=1}^\infty \alpha_n\frac{f(x_n)-f(y_n)}{d(x_n,y_n)},\qquad f\in\Lipo(X),
$$
where $\{\alpha_n\}\in\ell_1$ and $\{(x_n,y_n)\}\in \widetilde{X}^\N$.
\end{corollary}


\begin{thebibliography}{99}
\bibitem{as} R. M. Aron and M. Schottenloher, Compact holomorphic mappings on Banach spaces and the approximation property, J. Funct. Anal. \textbf{21} (1976) 7--30.  
\bibitem{bm} K. D. Bierstedt and R. Meise, Bemerkungen über die Approximationseigenschaft lokalkonvexer Funktionenräume, Math. Ann. \textbf{209} (1974) 99--107.
\bibitem{b} L. Borel-Mathurin, Approximation properties and non-linear geometry of Banach spaces, Houston J. Math. \textbf{38}, no. 4, (2012) 1135--1148.
\bibitem{bu} R. C. Buck, Operator algebras and dual spaces, Proc. Amer. Math. Soc. \textbf{3} (1952) 681--687.
\bibitem{c} J. B. Cooper, Saks spaces and applications to functional analysis. Second edition. North-Holland Mathematics Studies, 139. Notas de Matemática [Mathematical Notes], 116. North-Holland Publishing Co., Amsterdam, 1987.
\bibitem{dm} S. Dineen and J. Mujica, The approximation property for spaces of holomorphic functions on infinite dimensional spaces I, J. Approx. Theory \textbf{126} (2004) 141--156.
\bibitem{du} J. Diestel and J. J. Uhl, Jr., Vector measures. With a foreword by B. J. Pettis. Mathematical Surveys, No. 15. American Mathematical Society, Providence, R.I., 1977.
\bibitem{ds} N. Dunford and J. T. Schwartz, Linear operators. I: General theory, Pure and Appl. Math., vol. 7, Interscience, New York, 1958. 
\bibitem{e} P. Enflo, A counterexample to the approximation property, Acta Math. \textbf{130} (1973) 309-317.
\bibitem{gk} G. Godefroy and N. J. Kalton, Lipschitz-free Banach spaces, Studia Math. \textbf{159} (2003) 121--141.
\bibitem{go} G. Godefroy and N. Ozawa, Free Banach spaces and the approximation properties, Proc. Amer. Math. Soc. \textbf{142}, no. 5, (2014) 1681--1687.
\bibitem{g} A. Grothendieck, Produits tensoriels topologiques et espaces nucl\'{e}aires, Mem. Amer. Math. Soc. 16 (1955).
\bibitem{hp} P. H\'{a}jek and E. Perneck\'{a}, On Schauder bases in Lipschitz-free spaces, J. Math. Anal. Appl. \textbf{416} (2014) 629--646.
\bibitem{h} R. B. Holmes, Geometric functional analysis and its applications, Graduate Texts in Mathematics, No. 24. Springer-Verlag, New York-Heidelberg, 1975.
\bibitem{jsv} A. Jim\'{e}nez-Vargas, J. M. Sepulcre and Mois\'{e}s Villegas-Vallecillos, Lipschitz compact operators, J. Math. Anal. Appl. \textbf{415} (2014) 889--901.
\bibitem{j70}  J. A. Johnson, Banach spaces of Lipschitz functions and vector-valued Lipschitz functions, Trans. Amer. Math. Soc. \textbf{148} (1970) 147--169.
\bibitem{j75}  J. A. Johnson, A note on Banach spaces of Lipschitz functions, Pacific J. Math. \textbf{58} (1975) 475--482.
\bibitem{k04} N. J. Kalton, Spaces of Lipschitz and H\"{o}lder functions and their applications, Collect. Math. \textbf{55} (2004) 171--217.
\bibitem{k} J. L. Kelley, General topology, Graduate Texts in Mathematics, No. 27. Springer-Verlag, New York-Berlin, 1975.
\bibitem{lp} G. Lancien and E. Perneck\'{a}, Approximation properties and Schauder decompositions in Lipschitz-free spaces, J. Funct. Anal. \textbf{264} (2013) 2323--2334.
\bibitem{m} J. Mujica, Linearization of bounded holomorphic mappings on Banach spaces, Trans. Amer. Math. Soc. \textbf{324} (1991) 867--887.
\bibitem{ryan} Raymond A. Ryan, Introduction to Tensor Products of Banach Spaces, Series: Springer Monographs in Mathematics, Springer, 2002.
\bibitem{s1} L. Schwartz, Produits tensoriels topologiques d'espaces vectoriels topologiques, Applications, Séminaire Schwartz, 1953--1954.
\bibitem{s} L. Schwartz, Th\'{e}orie des distributions \`{a} valeurs vectorielles I, Ann. Inst. Fourier (Grenoble), 7 (1957), 1--141.
\bibitem{weaver} N. Weaver, Lipschitz Algebras, World Scientific Publishing Co., Singapore, 1999.
\end{thebibliography}
\end{document}